\documentclass[a4paper]{amsart}

\usepackage{amsmath,amsthm,amssymb, color}
\usepackage[all]{xy}

\usepackage{graphicx,psfrag,epsfig}

\theoremstyle{plain}
\newtheorem{theorem}{Theorem}
\newtheorem{lemma}{Lemma}
\newtheorem{corollary}{Corollary}
\newtheorem{proposition}{Proposition}
\newtheorem{conjecture}{Conjecture}
\theoremstyle{definition}
\newtheorem{definition}{Definition}
\newtheorem{remark}{Remark}
\newtheorem{question}{Question}

\def\bdef{\begin{definition}}
\def\endef{\end{definition}}
\def\bthm{\begin{theorem}}
\def\ethm{\end{theorem}}
\def\blm{\begin{lemma}}
\def\elm{\end{lemma}}
\def\brm{\begin{remark}}
\def\erm{\end{remark}}
\def\bprop{\begin{proposition}}
\def\eprop{\end{proposition}}
\def\bcor{\begin{corollary}}
\def\ecor{\end{corollary}}
\def\be{\begin{eqnarray}}
\def\ee{\end{eqnarray}}
\def\beal{\begin{aligned}}
\def\enal{\end{aligned}}

\def\om{\omega}
\def\Om{\Omega}
\def\al{\alpha}
\def\eps{\varepsilon}
\def\phi{\varphi}
\def\f{\varphi}
\def\Id{\mathbb I}
\def\A{\mathbb A}
\def\R{\mathbb R}

\def\T{\mathbb T}
\def\Q{\mathbb Q}
\def\Z{\mathbb Z}

\def\L{\mathcal L}
\def\M{\mathcal M}

\def\~{\tilde}
\def\a{\alpha}
\def\gm{\gamma}
\def\g{\gamma}

\def\th{\theta}
\def\dt{\delta}
\def\lb{\lambda}
\def\Lb{\Lambda}
\def\cE{\mathcal E}

\def\cD{\mathcal D}

\def\be{\begin{equation}}
\def\ee{\end{equation}}
\def\bdef{\begin{definition}}
\def\endef{\end{definition}}
\def\blm{\begin{lemma}}
\def\elm{\end{lemma}}
\def\beal{\begin{aligned}}
\def\enal{\end{aligned}}

\newtheorem*{Pf}{Proof}
\renewenvironment{proof}{\begin{Pf} \begin{upshape}} {\end{upshape} \qed\end{Pf}}

\def\ff {\partial_\phi \phi'}
\def\sf {\partial_\phi s'}

\def\hs {{\hat{s}}}

\def\sg{\sigma}

\title{On conjugacy of convex billiards}
\author{Vadim Kaloshin \& Alfonso Sorrentino}
\address{Department of Mathematics, University of Maryland at College Park,
College Park, MD 20740, US.}
\email{vadim.kaloshin@gmail.com}
\address{Department of Pure Mathematics and Mathematical Statistics,
University of Cambridge, Cambridge, CB3 0WB, UK.}
\email{a.sorrentino@dpmms.cam.ac.uk}
\date{\today}
\begin{document}

\maketitle

\begin{abstract}
Given a strictly convex domain $\Om \subset \R^2$, there is a natural way to define a billiard map in it:
a rectilinear path hitting the boundary reflects so that the angle of reflection is equal to
the angle of incidence.
In this paper we answer a relatively old question of Guillemin.
We show that if two billiard maps are $C^{1,\a}$-conjugate near the boundary,
for some $\a>1/2$, then the corresponding domains are similar, {\it i.e.}
they can be obtained one from the other by a rescaling and an isometry.
As an application, we prove a conditional version of Birkhoff conjecture on the integrability of planar billiards and show that the original conjecture is equivalent to  what we call an {\it Extension problem}.\\
Quite interestingly, our result and a positive solution to this extension problem would provide an answer to a closely related question in spectral theory: if the marked length spectra of
two domains are the same, is it true that they are isometric?
\end{abstract}

\maketitle

\section{Introduction}

A {\it mathematical billiard} is a dynamical model describing the motion of a mass point
inside a (strictly) convex domain $\Omega \subset \R^2$ with smooth boundary.
The (massless) {billiard ball} moves with unit velocity and without friction
following a rectilinear path;  when it hits the boundary it reflects {elastically}
according to the standard {\it reflection law}: the angle of reflection is equal to
the angle of incidence. Such trajectories are sometimes called {\it broken geodesics}.

This conceptually simple model, yet mathematically complicated, has been first introduced
by G. D. Birkhoff \cite{Birkhoff} as a mathematical playground to prove, with as little
technicality as possible, some dynamical applications of Poincare's last geometric
theorem and its generalisations:\\

\begin{quote}
``[...]{\it This example is very illuminating for the following reason: Any dynamical
system with two degrees of freedom is isomorphic with the motion of a particle on
a smooth surface rotating uniformly about a fixed axis and carrying a conservative
field of force with it} (see \cite{Birkhoff1917}). {\it In particular if the surface
is not rotating and if the field of force is lacking, the paths of the particles
will be geodesics. If the surface is conceived of as convex to begin with and then
gradually to be flattened to the form of a plane convex curve $C$, the ``billiard ball''
problems results. But in this problem the formal side, usually so formidable in dynamics,
almost completely disappears, and only the interesting qualitative questions need
to be considered.}[...] ''\\
(G. D. Birkhoff, \cite[pp. 155-156]{Birkhoff})\\
\end{quote}

Since then, billiards have captured many mathematicians' attention and have slowly
become a very popular subject. Despite their apparently simple (local) dynamics, in fact,
their qualitative dynamical properties are extremely {non-local}! This global influence
on the dynamics translates into several intriguing {\it rigidity  phenomena}, which
are at the basis of many unanswered questions and conjectures.\\

\subsection{Conjugate billiard maps.}

The starting point of this work is the following question attributed to Victor Guillemin:

\begin{question} \label{q1}
{\it  Let $f$ and $g$ be smooth Birkhoff billiard maps corresponding to two strictly
convex domains $\Omega_f$ and $\Omega_g$. Assume that  $f$ and $g$ are conjugate,
i.e there exists $h$ such that $f = h^{-1} \circ g \circ h$. What can we say about
the two domains $\Omega_f$ and $\Omega_g$? Are they ``similar'', that is, have
they the same shape?}\\
\end{question}

To our knowledge the only known answer to this question is in the case of circular
billiards. A billiard map in a disc $\cD$, in fact, enjoys the peculiar property of
having  the phase space completely foliated by homotopically non-trivial invariant curves.
It is an example of so-called {\it integrable billiards}. In terms of the geometry
of the billiard domain $\cD$, this reflects the existence of a smooth foliation by
(smooth and convex){\it caustics}, {\it i.e.} (smooth and convex) curves with the property
that if a trajectory is tangent to one of them, then it will remain tangent after each
reflection. It is easy to check that in the circular billiard case this family consists
of concentric circles (figure \ref{circle-billiard}). See for instance \cite[Chapter 2]{Tabach}.\\

\begin{figure} [h!]
\begin{center}
\includegraphics[scale=0.3]{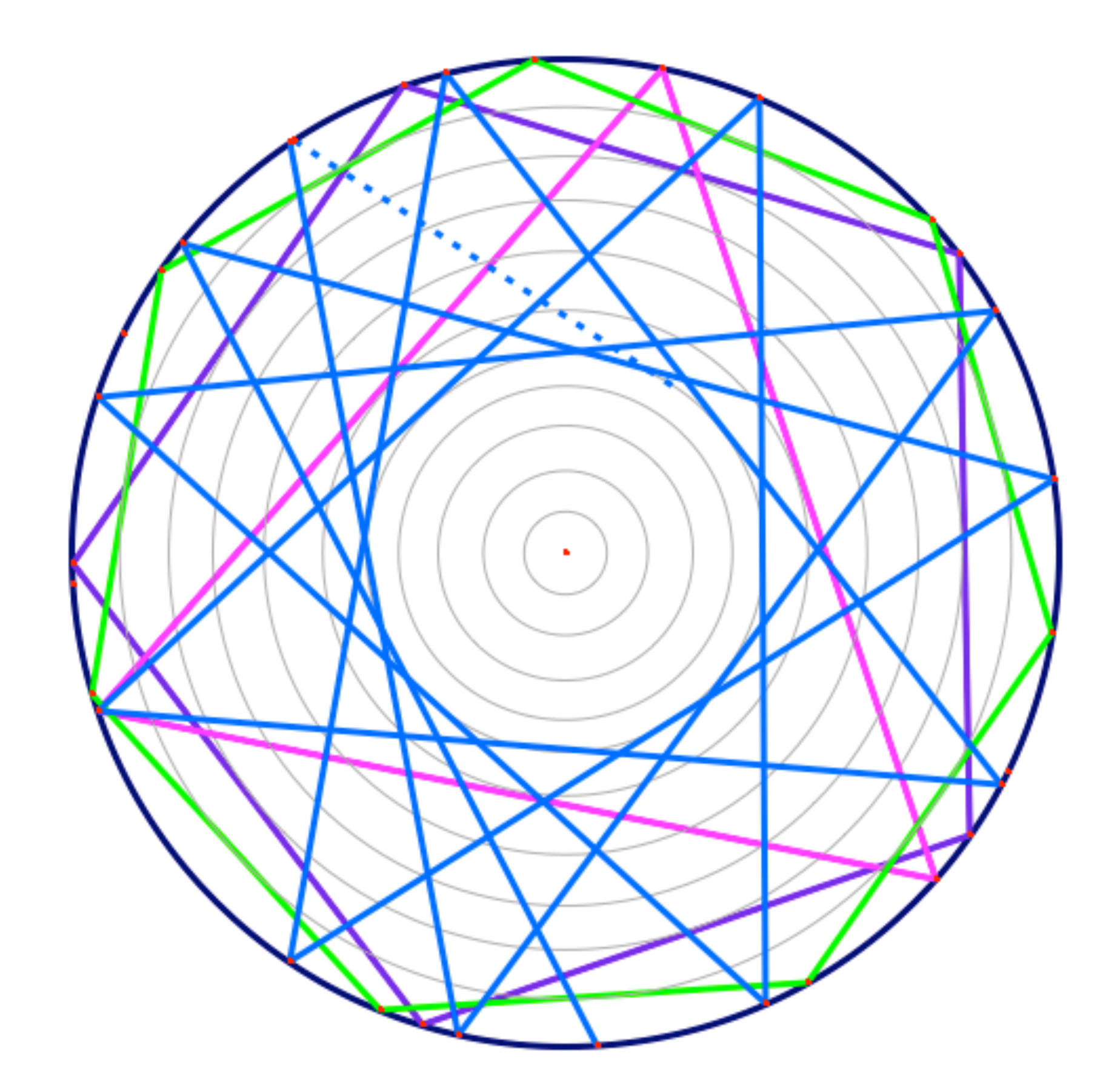}
\caption{}
\label{circle-billiard}
\end{center}
\end{figure}

In \cite{Bialy},  Misha Bialy proved the following beautiful result:\\

\noindent {\bf Theorem (Bialy).}
\noindent {\it If the phase space of the billiard ball map is foliated by continuous
invariant curves which are not null-homotopic, then  it is a circular billiard.}\\

Since having such a foliation is invariant under conjugacy, then Question \ref{q1}
has an affirmative answer if one of the two domains is a circle.\\

In this article we want to deal with a much weaker assumption:

\begin{question}\label{q2}
{\it What happens if we assume that the two billiard maps are conjugate
{\bf only} in a neighborhood of the boundary}?
\end{question}

Our interest in this variant of Question \ref{q1} comes from the following observation.
Circular billiards are not the only examples of integrable billiards.
Although there is no commonly agreed
definition, this can be understood either as the existence of an
{\it integral of motion} or as the presence of a (smooth) family of
caustics that foliate a neighborhood of the boundary (the equivalence
of this two notions is an interesting problem itself).
Billiards in ellipses, therefore, are also integrable. Yet, the dynamical picture
is very distinct from the circular case: as it is showed in figure \ref{ellipse-billiard}, each trajectory which does not pass through
a focal point, is always tangent to precisely one confocal conic section, either
a confocal ellipse or the two branches of a confocal hyperbola (see for example
\cite[Chapter 4]{Tabach}). Thus, the confocal ellipses inside an elliptical billiards
are convex caustics, but they do not foliate the whole domain: the segment between
the two foci is left out. \\
Hence, while it is clear from Bialy's result that a global conjugacy between the billiard maps in a non-circular ellipse and a circle cannot exist, it is everything
but obvious whether a {\it local} conjugacy near the boundary exists or does not.\\
See also Section \ref{secBC}.\\

\begin{figure} [h!]
\begin{center}
\includegraphics[scale=0.2]{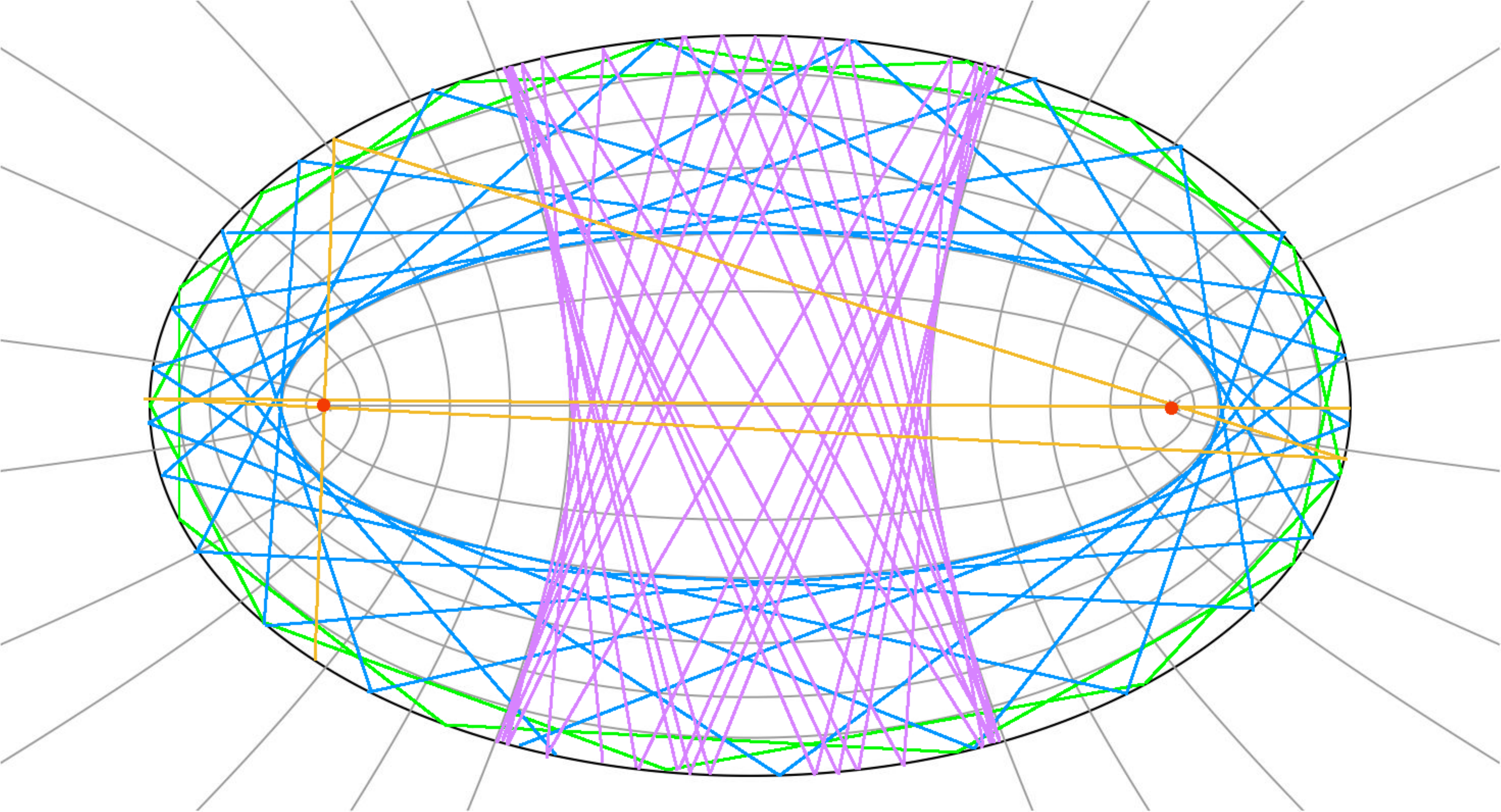}
\caption{}
\label{ellipse-billiard}
\end{center}
\end{figure}

In this article we prove the following result. \\

\noindent{\bf Main Theorem (Theorem \ref{maintheorem}).}
{\it Let $f$ and $g$ be billiard maps corresponding, respectively,
to strictly convex $C^r$ planar domains  $\Om_f$ and $\Om_g$, with $r\ge 4$.
Suppose that $h$ is a $C^{1,\a}$ conjugacy near the boundary with $\a>1/2$,
{\it i.e.} there exists $\dt>0$ such that $f=h^{-1} {g} h$ for any
$s \in \partial \Om_f$ and $0\le \phi\le \dt$.  Then the two domains are
similar, that is they are the same up to a rescaling and an isometry.}

\begin{remark}\label{rem1}
This provides an answer to Questions \ref{q1} and \ref{q2}, under very mild
smoothness assumptions on the conjugacy $h$. Furthermore, as we discuss in Appendix \ref{extension},
it is sufficient to have a $C^{1,\a}$ conjugacy $h$ {\it only }on an invariant
Cantor set which includes the boundary. This invariant Cantor set, for instance, could consist of invariant curves
(caustics) accumulating to the boundary. In the case of smooth strictly convex billiards, the existence of such invariant curves follows from
a famous result of Lazutkin \cite{Lazutkin} (see also \cite{KP}).
\end{remark}

Observe that if we assumed $h$ to be only continuous, then the answer to the above questions  would be negative.
In fact, as pointed out to us by John Mather, billiards in ellipses of different eccentricities cannot be globally  $C^0$--conjugate, due to dynamics at the elliptic periodic points. This provides a negative answer to Question \ref{q1}.
As for Question
\ref{q2}, we can prove the following:\\

\noindent {\bf Proposition \ref{C0conjugacy}}
{\it Let $\Om_f=\mathbb S^1$ be the circle  and $\Om_g$
be an ellipse of non-zero eccentricity. Then near
the boundary their billiard maps are $C^0$-conjugate.

Similary, if $\Om_f$ and $\Om_g$
are  ellipses of different eccentricities, then
their billiard maps are $C^0$-conjugate near the boundary.}\\

\subsection{Birkhoff conjecture}\label{secBC}
Let us now discuss the implications of our result to a famous conjecture
attributed to Birkhoff. Birkhoff conjecture states that amongst all
convex billiards, the only {\it integrable} ones are the ones in ellipses (a circle is a distinct special case).

Despite its long history and the amount of attention that
this conjecture has captured, it still remains essentially open.
As far as our understanding of integrable billiards is concerned,
the two most important related results are
the above--mentioned theorem by Bialy \cite{Bialy} and
a theorem by Mather \cite{Mather82} which proves the non-existence
of caustics if the curvature of the boundary vanishes at one point.
This latter justifies the restriction of our attention  to strictly convex domains.\\

The main result in this paper implies the following {\it Conditional Birkhoff
conjecture}:\\

\noindent{\bf Corollary (Conditional Birkhoff conjecture).}
{\it Let $\a>1/2$. If an integrable billiard is
$C^{1,\a}$-conjugate to an ellipse (resp. a circle) in a neighborhood of
the boundary, then it is an ellipse (resp. a circle).}\\

\vspace{10 pt}

In the light of Remark \ref{rem1} (see also Corollary \ref{main-corollary}), it would suffice
to have such a conjugacy only on a Cantor set of caustics which includes the boundary.
Observe that it follows from Lazutkin's result \cite{Lazutkin}
(see also \cite{KP}) that such
a conjugacy always exists on a Cantor set of caustics that
accumulates on the boundary, but does not include it (see Theorem \ref{Kovachev-Popov}).

The problem becomes understanding when this conjugacy
can be extended up to the boundary and what is its smoothness
(in Whitney's sense).  We state the following conjecture. \\

\begin{conjecture}\label{conj}
{\it  Let $\Omega$ be a smooth integrable billiard domain and
let $\cE$ be an ellipse with the same perimeter and the same Lazutkin
perimeter {\rm (see Section \ref{sec2})}. Then, the above-mentioned conjugacy can be smoothly extended (in Whitney's sense) up
to the boundary.}  \\
\end{conjecture}

This conjecture is closely related to the  {\it Extension problem} that we shall discuss more thoroughly
in Appendix \ref{extension}.\\

An interesting remark is the following.

\begin{corollary}
Birkhoff Conjecture is true if and only if
Conjecture \ref{conj} is true.
 \end{corollary}

\begin{proof}
($   \Longrightarrow$) It follows from Birkhoff's conjecture that
$\Omega=\cE$,  for some ellipse $\cE$. Of course $\cE$ will have
the same perimeter and Lazutkin perimeter as $\Omega$. Therefore
it is sufficient to consider the identity map as a possible conjugacy.\\
($\Longleftarrow$) If Conjecture \ref{conj} is true, then
we can apply our Conditional version of Birkhoff conjecture (and Corollary \ref{main-corollary}) to deduce
that $\Omega$ is an ellipse.\\
\end{proof}

\subsection{Length spectrum} \label{sec1.3}
Finally, we would like to describe how our result is related to a classical spectral problem:  {\it Can one hear the shape of a drum?}, as formulated in a very suggestive way by M. Kac \cite{Kac}. 
In the context of billiards (which can be considered as limit geodesic flow), this question becomes: what information on the geometry of the billiard domain are encoded in the length spectrum of its closed orbits?\\

Let us recall some basic definitions. The rotation number of a periodic billiard trajectory
(respectively, a closed broken geodesic) is a rational
number
\[
\dfrac{p}{q}
\ =\
\dfrac{\text{winding number}}
{\text{number of reflections}}\ \in\ \big(0,\frac 12\Big],
\]
where the winding number $p>1$ is defined as follows (see Section \ref{sec2} for a more precise definition).
Fix the positive orientation of $\partial \Om$ and
pick any reflection point of the closed geodesic on �
$\partial \Om$; then follow the trajectory and count
how many times it goes around $\partial \Om$ in
the positive direction until it comes back to
the starting point.
Notice that inverting the direction of motion for every
periodic billiard trajectory  of rotation number
$p/q  \in (0, 1/2]$, we obtain an orbit of rotation number
$(q-p)/q \in [1/2,1)$. \\

In \cite{Birkhoff}, Birkhoff proved that for every $p/q \in (0, 1/2]$ in lowest terms,
there are at least two closed geodesics of rotation number
$p/q$: one maximizing the total length and the other obtained by min-max methods (see also  \cite[Theorem 1.2.4]{Siburg}).\\

The {\it length spectrum} of $\Om$ is defined as the set
\[
\L_\Om:=\mathbb N\{\text{ lengths of closed geodesics in }\Om\}
\cup \mathbb N \ell(\partial \Omega),
\]

where $\ell(\partial \Omega)$ denotes the length of the boundary.

One can refine this set of information in a more useful way.
For each rotation number $p/q$ in lowest terms, let us consider the maximal
length of closed geodesics having rotation number $p/q$ and
``label'' it using the rotation number.  This map is called the {\it marked length spectrum} of $\Omega$:
\[
\M\L_\Om:\Q\cap\big(0,\frac 12\Big] \longrightarrow \R_+
\]

This map is closely related to Mather's minimal average action (or $\beta$-function). See Appendix \ref{beta}.\\

One can ask the following questions,  which are related to a well-known conjecture by Guillemin and Melrose.\\

\begin{question}  \label{q3}{\it Let $\Om$ and $\Om'$ be two strictly
convex $C^{6}$ domains with the same length spectra
$\M\L_{\Om}$ and $\M\L_{\Om'}$. Are  the two domains
$\Om$ and $\Om'$ isometric?}
\end{question}

The same question might be asked for marked length spectrum.\\

As we will discuss in Appendix \ref{extension}, It follows from our main result and Corollary \ref{main-corollary} that a positive answer to this question
for strictly convex domains relies on a solution to an extension problem, similar to the one stated in Conjecture \ref{conj}. See Appendix \ref{extension} for more details.\\

\begin{remark}
A remarkable relation exists between the length spectrum
of a billiard in a convex domain $\Omega$  and the spectrum of the Laplace operator
in $\Omega$ with Dirichlet boundary condition:
\[
\left\{
\begin{array}{l}
\Delta f = \lb f \quad \text{in}\; \Om \\
f|_{\partial \Om} = 0.\end{array}\right.\\
\]

From the physical point of view, the eigenvalues $\lb$
are the eigenfrequencies of the membrane $\Om$ with
a fixed boundary. Denote by $\Delta_\Om$ the Laplace
spectrum of eigenvalues solving this problem.
The famous question of M. Kac in its original version asks {\it if one can
recover the domain from the Laplace spectrum}.
For general manifolds there are counterexamples (see
\cite{GordonWebbWolpert}).\\

K. Anderson and R. Melrose \cite{AM}
proved the following relation between the Laplace spectrum and the length spectrum:\\

\noindent {\bf Theorem (Anderson-Melrose).}
{\it The sum
\[
\sum_{\lb_i \in spec \Delta}
cos (t \sqrt{-\lb_i})
\]
is a well-defined generalized function (distribution) of $t$,
smooth away from the length spectrum. That is, if $l > 0$
belongs to the singular support of this distribution, then
there exists either a closed billiard trajectory of length l,
or a closed geodesic of length l in the boundary of
the billiard table.}\\

\end{remark}

\vspace{20 pt}

\subsection{Outline of the article}
The article is organized as follows. In Section \ref{sec2} we discuss some properties of the billiard map and describe its Taylor expansion near the boundary. The proof of this result will be postponed to Appendix \ref{appendix}. In Section \ref{sec3} we prove the main results stated in the Introduction. In Appendix \ref{extension} we discuss a version of Lazutkin's famous result on the existence of smooth caustics which accumulate on the boundary. This will allow us to state an ``Extension problem'', whose conjecturally positive solution is related to Birkhoff conjecture and Guillemin-Melrose's problem (see the above discussion in sections \ref{secBC} and \ref{sec1.3}). Finally, in Appendix \ref{appLaz} we prove some technical results on Lazutkin coordinates, that are used in the proof of the main results.

\vspace{20 pt}

\noindent{\sc Acknowledgements.}
The authors thank John Mather for pointing out that
billiards in ellipses of different eccentricities are
not globally $C^0$-conjugate, due to dynamics at the elliptic
periodic points.  The authors express also gratitude to Peter Sarnak and Sergei Tabachnikov for several interesting remarks. \\

%%%%%%%%%%%%%%%%%%%%%%%%%%%%%%%%%
\section{The billiard map}\label{sec2}

In this section we would like to recall some properties of the billiard map. We refer to \cite{Siburg, Tabach} for a more comprehensive introduction to the study of billiards.\\

Let $\Omega$ be a strictly convex domain in $\R^2$ with $C^r$ boundary $\partial \Omega$,
with $r\geq 3$. The phase space $M$ of the billiard map consists of unit vectors
$(x,v)$ whose foot points $x$ are on $\partial \Omega$ and which have inward directions.
The billiard ball map $f:M \longrightarrow M$ takes $(x,v)$ to $(x',v')$, where $x'$
represents the point where the trajectory starting at $x$ with velocity $v$ hits the boundary
$\partial \Omega$ again, and $v'$ is the {\it reflected velocity}, according to
the standard reflection law: angle of incidence is equal to the angle of reflection (figure \ref{billiard}).

\begin{figure} [h!]
\begin{center}
\includegraphics[scale=0.35]{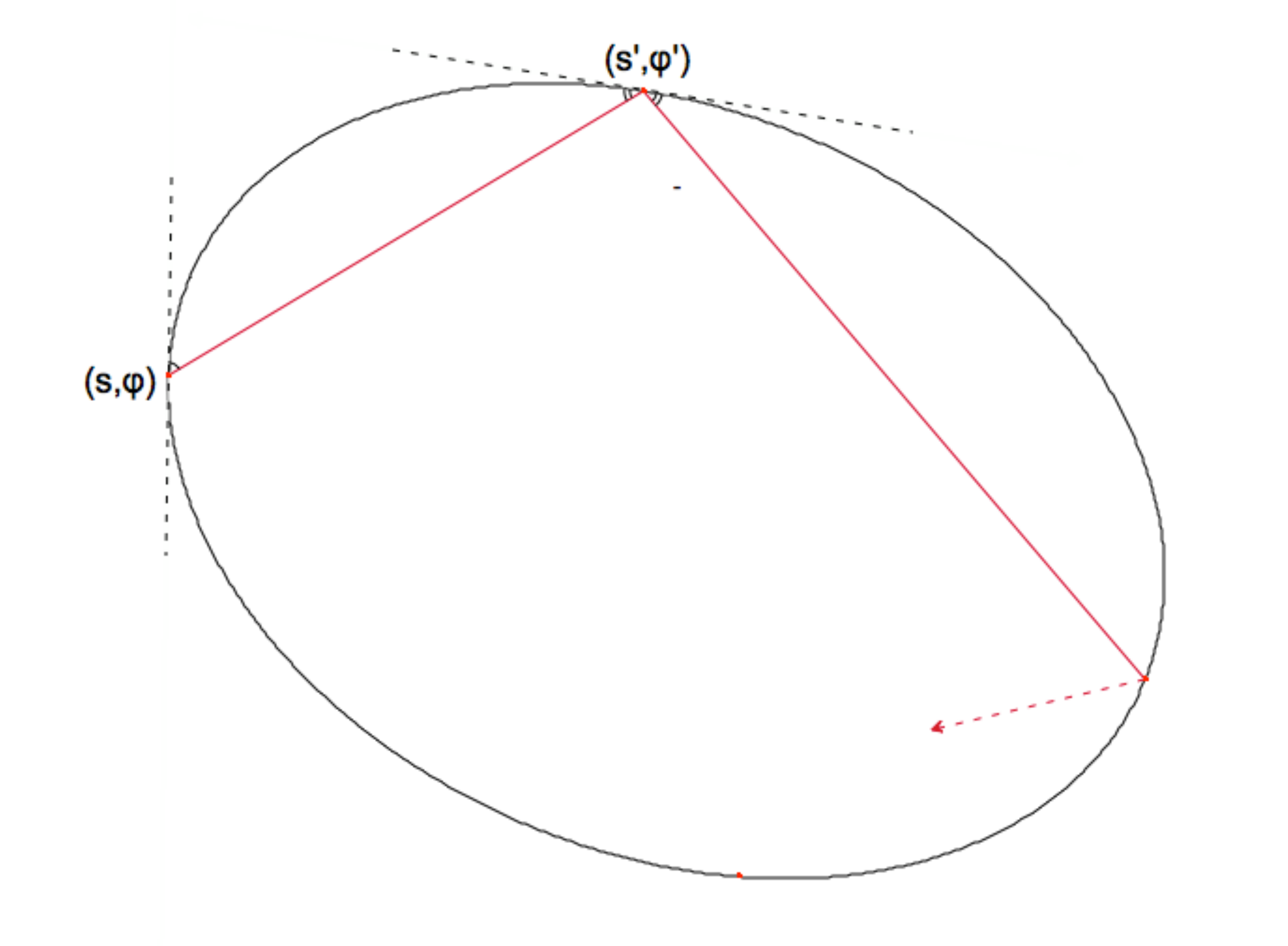}
\caption{}
\label{billiard}
\end{center}
\end{figure}

\begin{remark}
Observe that if $\Omega$ is not convex, then the billiard map is not continuous.
Moreover, as pointed out by Halpern \cite{Halpern}, if the boundary is not at
least $C^3$, then the flow might not be complete.
\end{remark}

Let us introduce coordinates on $M$.
We suppose that $\partial \Omega$ is parametrized  by  arc-length $s$ and
let $\g:  [0, l] \longrightarrow \R^2$ denote such a parametrization,
where $l=l(\partial \Omega)$ denotes the length of $\partial \Omega$. Let $\phi$
be the angle between $v$ and the positive tangent to $\partial \Omega$ at $x$.
Hence, $ M$ can be identified with the annulus $\A = [0,l] \times (0,\pi)$
and the billiard map $f$ can be described as

\begin{eqnarray*}
f: [0,l] \times (0,\pi) &\longrightarrow& [0,l] \times (0,\pi)\\
(s,\phi) &\longmapsto & (s',\phi').
\end{eqnarray*}

In particular $f$ can be extended to $\bar{\A}=[0,l] \times [0,\pi]$ by fixing $f(s,0)=f(s,\pi)= s$
for all $s$. \\

Let us denote by
$$
\ell(s,s') := \|\g(s) - \g(s')\|
$$
the Euclidean distance between two points on $\partial \Omega$. It is easy to prove that
\begin{equation}
\left\{ \begin{array}{l}
\dfrac{\partial \ell}{\partial s}(s,s') = - \cos \phi \\
\\
\dfrac{\partial \ell}{\partial s'}(s,s') = \cos \phi'\,.\\
\end{array}\right.
\end{equation}

\begin{remark}
If we lift everything to the universal cover and introduce new coordinates
$(\tilde{s},r)=(s, \cos \phi) \in \R \times (-1,1)$,then the billiard map is a twist map
with $\ell$ as generating function. See \cite{Siburg, Tabach}.\\
\end{remark}

Particularly interesting billiard orbits are {\it periodic orbits}, {\it i.e.} billiard orbits 
$X=\{x_k\}_{k\in\Z}:=\{(s_k,\phi_k)\}_{k\in \Z}$ for which there exists an integer $q\geq 2$ such that
$x_k=x_{k+q}$ for all $k\in \Z$. The minimal of such $q's$ represents the {\it period} of the orbit.
However periodic orbits with the same period may be of very different topological types. A useful topological invariant that allows us to distinguish amongst them is the so-called {\it rotation number}, which can be easily defined as follows. Let $X$ be a periodic orbit of period $q$ and consider the corresponding $q$-tuple $(s_1,\ldots, s_q) \in \R/l\Z$. For all $1\leq k \leq q$, there exists $\lambda_k \in (0,l)$ such that $s_{k+1}=s_k+\lb_k$  (using  the periodicity, $s_{q+1}=s_1$). Since the orbit is periodic, then $ \lb_1 +\ldots + \lb_k \in l\Z$ and takes value between $l$ and $(q-1)l$. The integer $p:= \frac{\lb_1 +\ldots + \lb_k}{l}$ is called the {\it winding number} of the orbit. The rotation number of $X$ will then be the rational number $\rho(X):=\frac{p}{q}$.
Observe that changing the orientation of the orbit replaces the rotation number $\frac{p}{q}$ by 
$\frac{q-p}{q}$. Since we do not distinguish between two opposite orientations, we can assume that $\rho(X) \in (0,\frac{1}{2}\big] \cap \Q$.\\

In \cite{Birkhoff}, as an application of Poincare's last geometric theorem, Birkhoff proved the following result.\\

\noindent{\bf Theorem [Birkhoff]}
{\it For every $p/q \in (0, 1/2]$ in lowest terms, there are at least two geometrically distinct periodic billiard trajectories with rotation number $p/q$}.\\

\vspace{10 pt}

Another important aspect in the study of billiard maps  is the role played by the curvature of the boundary in determining the local and global dynamics.  
The proof of our main results will strongly rely on the following formula for $Df$, 
which consists of  a Taylor expansion of $Df$ near the boundary ({\it i.e.} for small $\phi$) and provides an explicit description of the coeffiecients in terms of the curvature. 
We believe that this might well be of independent interest. \\

\begin{proposition}\label{propformbillmap}
Let $\Omega$ be a strictly
convex domain with  $C^4$ boundary. Let  $\rho(s)$ be the curvature of
the boundary at point $\gamma(s)$. Then, for small $\phi$, the differential
of the billiard map has the following form
\[
Df(s,\phi)=L(s)+\phi A(s)+ O(\phi^2)
\]
where
\be
L(s):=\left( \beal \quad 1 \quad & \quad \frac {2}{\rho(s)} \
\\  \quad 0 \quad & \qquad 1 \enal \right)
\ee
and
\be
A(s):=
\left(
\beal
\qquad  - \frac 43 \frac{\rho'(s)}{\rho^2(s)} \qquad
 &  \qquad - \frac 83 \frac{\rho'(s)}{\rho^3(s)} \qquad \\
\qquad - \frac 23 \frac{\rho'(s)}{\rho (s)}   \qquad  &
\qquad \ \frac 43 \frac{\rho'(s)}{\rho^2(s)} \qquad
\enal
\right).
\ee
\end{proposition}
%%%%%%%%%%%%%%%%%%%%%%
\vspace{10 pt}

We postpone the proof of  this Proposition to Appendix \ref{appendix}.\\

\vspace{20 pt}

Before concluding this section, we would like to recall another important result in the theory of billiards. In \cite{Lazutkin} V. Lazutkin introduced a very special
change of coordinates that reduces the billiard map $f$ to a very simple form.

Let $L_\Omega:  {[0,l]\times [0,\pi] \to  \T \times [0,\dt]}$  with small $\dt>0$ be given by
\begin{equation}\label{Lazutkincoord}
L_\Omega(s,\phi)=\left(x=C^{-1}_\Omega \int_0^s \rho^{-2/3}(s)ds,\qquad
y=4C_\Omega^{-1}\rho^{1/3}(s)\ \sin \phi/2 \right),
\end{equation}
where $C_\Omega := \int_0^l \rho^{-2/3} (s)ds$ is sometimes called the {\it Lazutkin perimeter}
(observe that it is chosen so that period of $x$ is one).

In these new coordinates the billiard map becomes very simple (see \cite{Lazutkin}):

\be \label{lazutkin-billiard-map}
f_L(x,y) = \Big( x+y +O(y^3),y + O(y^4) \Big)
\ee

In particular, near the boundary $\{\phi=0\} = \{y=0\}$, the billiard map $f_L$ reduces to
a small perturbation of the integrable map $(x,y)\longmapsto (x+y,y)$.

\begin{remark}
Using this result and KAM theorem, Lazutkin proved in \cite{Lazutkin}
that if $\partial \Omega$
is sufficiently smooth (smoothness is determined by KAM theorem), then there exists a positive measure set of caustics, which accumulates on the boundary and on which the motion is smoothly conjugate to a rigid rotation. See Appendix \ref{extension}.\\
\end{remark}

\section{Main results}\label{sec3}
We now prove the main results of this paper.
%\newpage
\bthm \label{maintheorem} Let $\Om_f$ and $\Om_g$ be two $C^4$ smooth
strictly convex domains and $\al>1/2$. If \ $\Om_f$ and $\Om_g$ are not similar,
then there cannot exist any $C^{1,\al}$ conjugacy near the boundary.

If billiard maps $f$ and $g$ can be $C^{1,\al}$ conjugate for some $\al>1/2$,
then $\Om_f$ and $\Om_g$ are similar.
\ethm

\bcor Let $\Om_f$ be the circle or an ellipse and $\Om_g$
be an ellipse of different eccentricity. Then near the boundary
their billiard maps are not $C^{1,\al}$-conjugate
for any $\al>1/2$.
\ecor

\begin{remark}\label{remmaithm}
We emphasize that the proof presented below works also
in the following framework. Let $\mathcal C_f$ and
$\mathcal C_g$ be two invariant Cantor sets consisting of families of
invariant curves and the boundaries (see Theorem \ref{Kovachev-Popov}).
Assume there exists a $C^{1,\a}$ conjugacy
$h:\mathcal C_f \to \mathcal C_g$, where it is smoothness in
the Cantor direction is understood in the sense of Whitney.
The reader should have in mind both settings.\\
\end{remark}

\noindent {\bf Outline of the proof.} The main ideas behind the proof can be summarized as follows.
\begin{itemize}
\item[{\Small{\sc Step 1:}}]  We use the conjugacy $f=h^{-1} {g} h$ to obtain equation (\ref{eqstep1}), which provides a relation between the differentials $Df$, $Dg$ and $Dh$.
\item[{\Small{\sc Step 2:}}]  Using Proposition \ref{propformbillmap}, we rewrite  $Df$, $Dg$, $Dh$ and $Dh^{-1}$, so to single out the terms of order $O(1)$, $O(\phi)$ and $o(\phi)$, as $\phi$ goes to zero.
\item[{\Small{\sc Step 3:}}] Using the results of Step 2 and the conjugacy equation  (\ref{eqstep1}) from Step 1, we obtain a more explicit matrix relation (\ref{eqstep2}) that will be the main object of our further investigation.
\item[{\Small{\sc Step 4:}}] We look at the $O(1)$ terms in the relation of Step 3 and deduce that $Dh(s,0)$ is upper-triangular.
\item[{\Small{\sc Step 5:}}] We look at the left lower entries in the relation from Step 3. It follows from Step 4 that there are no terms of $O(1)$. We study then terms of $O(\phi)$. This is quite an involved part of the proof, that we split in several parts (\ref{equality0}-\ref{equality4}). As a final result, we obtain relation (\ref{eqstep5}).
\item[{\Small{\sc Step 6:}}] The main part of this step is to prove that in relation (\ref{eqstep5}) from Step 5, there cannot be terms of $O(\phi)$. This implies that the curvatures of the domains, $\rho_f$ and $\rho_g$,  must satisfy a functional equation (\ref{eqstep6}).
\item[{\Small{\sc Step 7:}}] The final step of the proof consists in proving that  the functional equation from Step 6 forces  the curvatures to be the same (up to a rescaling). This concludes the proof of the theorem.\\
\end{itemize}

\begin{proof}
\underline{{\sc Step 1}}. Suppose that there exists $h:\T \times [0,\dt]\to \A$, which is
a $C^1$ conjugacy near the boundary, {\it i.e.} for some $\dt>0$
we have $f=h^{-1} {g} h$ for any $s\in \partial \Om_f$ and $0\le \phi\le \dt$.

Consider the corresponding equation on differentials:
\begin{eqnarray*}
 Df(s,\phi) &=& Dh^{-1}(g(h(s,\phi))) \cdot Dg(h(s,\phi)) \cdot Dh(s,\phi) =\\
&=&   Dh_+(s,\phi)^{-1}\cdot Dg(h(s,\phi)) \cdot Dh(s,\phi),
\end{eqnarray*}
where we denoted $Dh_+(s,\phi)=Dh(f(s,\phi))$. 
Rewrite
\begin{equation}\label{eqstep1}
D{f}-\Id
=Dh_+^{-1}\, (D{g} -\Id)\, Dh+\left(
Dh^{-1}_{+} - Dh^{-1} \right) \cdot Dh.
\end{equation}

\vspace{20 pt}

\noindent \underline{{\sc Step 2}}.
First, observe that
\begin{eqnarray*}
h(s,\phi) &=& h(s,0) + Dh(s,0) \left( \begin{array}{l} 0\\ \phi\end{array}\right) + O(\phi^{1+\a}) =\\
&=& \Big( \hat{s} + \partial_\phi h_1(s,0)\phi, \partial_\phi h_2(s,0)\phi\Big) + O(\phi^{1+\a}) =\\
&=:& \Big( \hat{s} + a(s) \phi, b(s) \phi\Big) + O(\phi^{1+\a}),
\end{eqnarray*}
where $\hs=\hs(s)$ is uniquely defined by the relation $h(s,0)=(\hs,0)$.

Decompose $Dh(s,\phi)$ as follows:
\[
Dh(s,\phi)=\Lb(s)+ \Lb_1(s,\phi),
\]
where $\Lb(s)=Dh(s,0)$ and
$\Lb_1(s,\phi)=Dh(s,\phi)-\Lb(s)$.
By Proposition \ref{propformbillmap} near
the boundary the billiard maps have the form
\[
Df(s,\phi)=L_f(s)+\phi A_f(s)+ O(\phi^2)
\qquad
Dg(s,\phi)=L_g(s)+\phi A_g(s)+ O(\phi^2).
\]
Therefore, by Proposition \ref{propformbillmap} we have
\be
Dg(h(s,\phi)) =  L_g\big(\hat{s} + a(s)\phi\big)
+ b(s) A_g(\hat{s}) \phi  +O(\phi^{1+\a}),
\ee
where $L_g$ and $A_g$ explicitly given.

Observe that $L_g$ is upper-triangular and that
$A_g$ has no identically zero entries unless
domain is the circle. Similarly for $Df$.
Write $Dh^{-1}(s,\phi)=\Lb^{-1}(s)+\Lb'_1(s,\phi)$.
Let us compute $\Lb'_1(s,\phi)$:
\[
\Id=[\Lb(s)+\Lb_1(s,\phi)]\cdot [\Lb^{-1}(s)+\Lb'_1(s,\phi)]
= \]
\[
=\Id+\Lb_1(s,\phi)\Lb^{-1}(s)+
[\Lb(s)+\Lb_1(s,\phi)]\cdot\Lb'_1(s,\phi).
\]
This implies
\begin{equation}\label{Lb1'}
 \Lb'_1=- (\Lb+\Lb_1)^{-1}\cdot \Lb_1 \cdot \Lb^{-1}.
\end{equation}
Observe that $\Lb_1$ and $\Lb'_1$ vanish as $\phi\to 0$. In particular, since
$Dh$ is $C^\al$ in $\phi$, we have
\[
\Lb_1=O(\phi^\al)\qquad \text{ and }\qquad \Lb'_1=O(\phi^\al).
\]

Let $\pi_s(s,\phi)=s$ be the natural projection.
Denote  $Dh_+(s,\phi) = \Lb_+(s) + \Lb_{1+}(s,\phi)$, where
$\Lb_+(s,\phi)=\Lb_+(\pi_s(f(s,\phi)))$,
$\Lb'_{1+}(s,\phi)=\Lb'_{1}(f(s,\phi))$.

\vspace{20 pt}

\noindent \underline{{\sc Step 3}}.
The conjugacy condition
$$Df(s,\phi)=Dh_+^{-1}(s,\phi) \cdot Dg (h(s,\phi)) \cdot Dh(s,\phi)$$
can be rewritten as follows:

{\small
\begin{eqnarray*}
D{f}(s,\phi)-\Id &=& Dh_+^{-1}(s,\phi) \cdot \big(D{g}(h(s,\phi)) -\Id\big) \cdot Dh(s,\phi)+
\\ &+& \left(
Dh^{-1}_{+}(s,\phi) - Dh^{-1}(s,\phi) \right) \cdot Dh(s,\phi) =\\
&=& Dh_+^{-1}(s,\phi) \cdot \big(D{g}(h(s,\phi)) -\Id\big)\cdot Dh(s,\phi)+\\
& +&  \left[(\Lb^{-1}_+ - \Lb^{-1})+
(\Lb'_{1+} - \Lb_{{1}}') \right] \cdot Dh(s,\phi).
\end{eqnarray*}
}

Using the notations above and the formula for $L_f$ we have:
\begin{eqnarray} \label{eqstep2}
L_f(s)-\Id+\phi A_f(s,\phi)& + & O(\phi^2) \nonumber \\
&=&
\Big[ (\Lb^{-1}_+ - \Lb^{-1})+
(\Lb'_{1+} - \Lb_{{1}}') \Big] \cdot (\Lb+\Lb_1)+\nonumber \\
&& \ + \
(\Lb_+^{-1}+\Lb'_{1+}) \cdot \big(L_g(\hs + a(s)\phi ) -\Id \big)\cdot
(\Lb+\Lb_1) +\nonumber \\
&& \ +\  b(s) \phi (\Lb_+^{-1}+\Lb'_{1+}) \cdot A_g(\hs) \cdot
(\Lb+\Lb_1) + O(\phi^2) = \nonumber \\
%%%%%%%
&=&
\Big[ (\Lb^{-1}_+ - \Lb^{-1})+
(\Lb'_{1+} - \Lb_{{1}}') \Big] \cdot (\Lb+\Lb_1)+\\
&& \ + \
(\Lb_+^{-1}+\Lb'_{1+}) \cdot \big(L_g(\hs) -\Id \big)\cdot
(\Lb+\Lb_1) +\nonumber \\
&& \ + \  a(s) \phi
(\Lb_+^{-1}+\Lb'_{1+}) \cdot D_sL_g(\hs)\cdot
(\Lb+\Lb_1) +\nonumber \\
&& \ +\  b(s) \phi (\Lb_+^{-1}+\Lb'_{1+}) \cdot A_g(\hs) \cdot
(\Lb+\Lb_1) + O(\phi^2). \nonumber
\end{eqnarray}

\vspace{20 pt}

\noindent \underline{{\sc Step 4}}.
Comparing the two sides of the above equality, we obtain (in the limit $\phi \rightarrow 0$):
\begin{eqnarray*}
L_f(s)-\Id= \Lb^{-1}(s)\cdot (L_g(\hs)-\Id) \cdot \Lb(s).
\end{eqnarray*}

This implies that $\Lb(s)$ is upper-triangular.
Denote its entries by $\lb_{11}(s), \lb_{12}(s),\lb_{22}(s)$
depending on a position in the matrix. The above equality
gives
\[
\dfrac{\lb_{22}(s)}{\lb_{11}(s)}=\dfrac{\rho_g(\hat s)}{\rho_f(s)} := \lb^{-3}.
\]

\vspace{20 pt}

\noindent \underline{{\sc Step 5}}.
Rewrite the remaining part as follows. We consider all terms that are not $o(\phi)$
and use the fact that $(\Lb_+^{-1} - \Lb^{-1})= o(1)$.
{\small
\begin{eqnarray}
\label{equality0}&& \phi \ (A_f (s)-  b(s) \Lb^{-1}(s) \cdot A_g(\hs) \cdot \Lb(s)) + o(\phi) =  \\
\label{equality1}&& \ \ = \ \Big[ (\Lb^{-1}_+ - \Lb^{-1}) + (\Lb'_{1+} - \Lb_{1}') \Big] \cdot (\Lb +\Lb_1) + \\
\label{equality2}&& \ \ + \   a(s) \phi\Big[ \Lb_+^{-1} \cdot D_sL_g(\hs)  \cdot \Lb \Big] + \\
\label{equality3}&& \ \ + \  \Lb'_{1+} \cdot (L_g-\Id) \cdot \Lb + \Lb^{-1}_+ \cdot (L_g-\Id) \cdot \Lb_1+ \\
\label{equality4}&& \ \ + \  \Lb'_{1+} \cdot (L_g-\Id) \cdot \Lb_1  +o(\phi).
\end{eqnarray}
}

We want now to analyze the left lower entries of the above matrices.

\begin{itemize}
\item Consider the left lower entry on  {\bf line (\ref{equality0})}.
By Proposition \ref{propformbillmap}  and formula
for ratio of $\lb_{11}(s)/\lb_{22}(s)$ we have

\begin{eqnarray*}
&& \phi\Big[ - \frac{2}{3} \dfrac{\rho'_f(s)}{\rho_f(s)} + b(s) \frac{\lb_{11}(s)}{\lb_{22}(s)}
\frac{2}{3} \dfrac{\rho'_g(\hs)}{\rho_g(\hs)} \Big] = \\
&& \ = \
- \frac{2}{3} \phi\Big[ \dfrac{\rho'_f(s)}{\rho_f(s)} - C^{-1} \lb
\dfrac{\rho_f(s)}{\rho_g(\hs)} \dfrac{\rho'_g(\hs)}{\rho_g(\hs)} \Big] = \\
&& \ = \
- \frac{2}{3} \rho_f(s) \phi\Big[ - \frac{d}{ds} \left(\dfrac{1}{\rho_f(s)}\right) + C^{-1} \lb
\frac{d}{d\hs} \left(\dfrac{1}{\rho_g(\hs)}\right) \Big] = \\
&& \ = \
\frac{2}{3}  \rho_f(s)  \phi\Big[ \frac{d}{ds} \left(\dfrac{1}{\rho_f(s)}\right) - C^{-1} \lb  \dfrac{ds}{d\hs}
\frac{d}{ds} \left(\dfrac{1}{\rho_g(\hs(s))}\right) \Big] = \\
&& \ = \
\frac{2}{3} \rho_f(s)  \phi\Big[ \frac{d}{ds} \left(\dfrac{1}{\rho_f(s)}\right) - C^{-1} \lb  C\lb^2
\frac{d}{ds} \left(\dfrac{1}{\rho_g(\hs(s))}\right) \Big] = \\
&& \ = \
\frac{2}{3}  \rho_f(s)  \phi\Big[ \frac{d}{ds} \left(\dfrac{1}{\rho_f(s)}\right) - \lb^3
\frac{d}{ds} \left(\dfrac{1}{\rho_g(\hs(s))}\right) \Big] = \\
&& \ = \
\frac{2}{3}  \rho_f(s)  \phi\Big[ \frac{d}{ds} \left(\dfrac{1}{\rho_f(s)}\right) - \dfrac{\rho_f(s)}{\rho_g(\hs(s))} \frac{d}{d s} \left(\dfrac{1}{\rho_g(\hs(s))}\right) \Big] = \\
&& \ = \
\frac{2}{3}  \rho^2_f(s)  \phi\Big[ \dfrac{1}{\rho_f(s)} \frac{d}{ds} \left(\dfrac{1}{\rho_f(s)}\right) - \dfrac{1}{\rho_g(\hs(s))} \frac{d}{ds} \left(\dfrac{1}{\rho_g(\hs(s))}\right) \Big] = \\
&& \ = \
\frac{1}{3}  \rho^2_f(s)  \phi   \frac{d}{ds} \Big[ \dfrac{1}{\rho^2_g(\hs(s))}
- \dfrac{1}{\rho^2_f(s)}
 \Big] = :  \Delta (s) \phi.
\end{eqnarray*}

Since domains are not similar, this function $\Delta(s)$ is $C^2$ smooth and
not identically zero. Therefore,
\begin{equation}\label{sigmatau}
\exists \ \sg,\ \tau>0 \quad \mbox{and} \ s^* \quad
\mbox{such that}  \ \Delta(s)>\sg \ \mbox{ for any} \  |s-s^*|<2\tau.
\end{equation}

\item Now we turn to {\bf line (\ref{equality1})}.
Denote entries of $\Lb_1(s,\phi)$ by $\lb^1_{ij}(s,\phi),\ i,j=1,2$ and
entries of $\Lb'_1(s,\phi)$ by $\lb'_{ij}(s,\phi),\ i,j=1,2$.
Since $\Lb_1(s,\phi)\to 0$ as $\phi\to 0$, for any $\eps>0$
there is $\kappa>0$ such that all entries of $\Lb_1(s,\phi)$
are at most $\eps$ in absolute value for any $s\in \T$
and any $|\phi|<\kappa$.
Compute the lower left entry of the sum of matrices on line (\ref{equality1}).
It can be decomposed into two terms.
\[
\Big[ \Lb^{-1}_+ - \Lb^{-1} \Big] \cdot (\Lb +\Lb_1)
\qquad \& \qquad
\Big[\Lb'_{1+} - \Lb_{{1}}'\Big] \cdot (\Lb +\Lb_1).
\]
The first term is easy to compute. We get
\[
 \big(\lb^{-1}_{22}(\pi_s f(s,\phi)) - \lb^{-1}_{22}(s) \big)
 \lb^1_{21}(s,\phi)= o(\phi).
\]
Computing the other one needs preliminary computations.
As we have already noticed above in (\ref{Lb1'}),
$\Lb_1'=-(\Lb+\Lb_1)^{-1}\cdot \Lb_1 \cdot \Lb^{-1}$.
The matrix $(\Lb+\Lb_1)^{-1}$ has the form:
\[
\dfrac{1}{D} \left( \beal \ \ \ \lb_{22}+\lb_{22}^1 \quad & \qquad
-(\lb_{12}+\lb_{12}^1) \\
\\   \qquad - \lb_{21}^1 \quad &
\qquad \ \ \lb_{11}+\lb_{11}^1 \enal \right),
\]
where $D=(\lb_{11}+\lb_{11}^1)(\lb_{22}+\lb_{22}^1)-\lb_{21}^1
(\lb_{12}+\lb_{12}^1)$ is the determinant.
The product $\Lb_1 \cdot \Lb^{-1}$ has the form
\[
\left( \beal\ \  \dfrac{\lb_{11}^1}{\lb_{11}} \quad & \qquad
- \dfrac{\lb_{11}^1\lb_{12}}{\lb_{11}\lb_{22}}+
\dfrac{\lb_{12}^1}{\lb_{22}}\ \\
\\  \  \ \dfrac{\lb_{21}^1}{\lb_{11}} \quad & \qquad
- \dfrac{\lb_{21}^1\lb_{12}}{\lb_{11}\lb_{22}}+
\dfrac{\lb_{22}^1}{\lb_{22}}\ \ \enal \right).
\]
Thus, the bottom left and right terms of the product
$\Lb_1'=-(\Lb+\Lb_1)^{-1}\cdot \Lb_1 \cdot \Lb^{-1}$
have the form
\[
- \dfrac{\lb_{21}^1}{D} \left(
\dfrac{\lb_{11}^1}{\lb_{11}} -
\dfrac{\lb_{11}+\lb_{11}^1}{\lb_{11}} \right)
=  \dfrac{\lb_{21}^1}{(\lb_{11}+\lb_{11}^1)(\lb_{22}+\lb_{22}^1)-\lb_{21}^1
(\lb_{12}+\lb_{12}^1)}
\]
\[
   +\dfrac{1}{D} \left[ \dfrac{-\lb^1_{21}(\lb_{12}+\lb_{12}^1)+\lb^1_{22}(\lb_{11}+\lb_{11}^1)}{\lb_{22}} \right]
=:\tau_2.
\]
{Rewrite the first expression in the form}
\[
\dfrac{\lb_{21}^1}{(\lb_{11}+\lb_{11}^1)\lb_{22}}-\qquad \qquad \qquad
\qquad \qquad \qquad \]
\[
-\lb_{21}^1\,\dfrac{\lb_{11}\lb_{22}^1+\lb_{22}^1\lb_{11}^1-
\lb_{21}^1(\lb_{12}+\lb_{12}^1)}
{(\lb_{11}+\lb^1_{11})\lb_{22}
\Big[(\lb_{11}+\lb_{11}^1)(\lb_{22}+\lb_{22}^1)-\lb_{21}^1
(\lb_{12}+\lb_{12}^1)\Big]}=:
\]
\[ \qquad \qquad \qquad
\qquad \qquad \qquad =:\tau_0(s,\phi)+\tau_1(s,\phi),
\]
where $\lb_{ii}$ is evaluated at $s$ and $\lb_{ij}^1$ at $(s,\phi)$
with $i,j=1,2$. All functions $\tau_j(s,\phi), j=0,1,2$  vanish
at $\phi=0$ and are $C^{\al}$-H\"older with exponent $\al>1/2$.

Composing $(\Lb_{1+}'-\Lb_1')\cdot (\Lb+\Lb_1)$ we get the following
lower left entry:
\[
\Big[\tau_0(f(s,\phi))+\tau_1(f(s,\phi))-\tau_0(s,\phi)-\tau_1(s,\phi)\Big]
(\lb_{11}(s)+\lb_{11}^1(s,\phi))+
\]
\[
+ \Big[\tau_2(f(s,\phi))-\tau_2(s,\phi)\Big] \lb_{21}^1(s,\phi).
\]
For small $\phi$ the billiard map has the form
\[
f(s,\phi)=\left(s+\dfrac{2\phi}{\rho_f(s)}+O(\phi^2),\phi+O(\phi^2)\right).
\]
H\"older regularity implies that
\begin{eqnarray}
\label{equality5}
&& (\tau_0(f(s,\phi))-\tau_0(s,\phi))
(\lb_{11}(s)+\lb_{11}^1(s,\phi))=O(\phi^\al)
\\
&& \label{equality6}
(\tau_1(f(s,\phi))-\tau_1(s,\phi))
(\lb_{11}(s)+\lb_{11}^1(s,\phi))=O(\phi^{2\al})=o(\phi) \\
&&
(\tau_2(f(s,\phi))-\tau_2(s,\phi)) \lb_{21}^1(s,\phi)=O(\phi^{2\al})=o(\phi),
\end{eqnarray}
where we have used that that $\lb_{ij}^1(s,\phi)=O(\phi^a)$.

Notice that we can rewrite (\ref{equality5}) in the form
\[
\left(\dfrac{\lb_{21}^1(f(s,\phi))}{\lb_{22}(\pi_s(f(s,\phi))}-
\dfrac{\lb_{21}^1(s,\phi)}{\lb_{22}(s)}\right)+O(\phi^{2\al}).
\]

\item Since $DL_g$ has also only one non-zero entry (the upper right),
then it follows easily that the lower left entry of the matrix
on {\bf line (\ref{equality2})} is also zero.\\

\item Since $\Lb(s)$ is upper triangular, so is $\Lb^{-1}(s)$.
By definition $(L_g-\Id)$ has only one non-zero entry:
the upper right. This implies that the lower left entry
of the sum of matrices on {\bf line (\ref{equality3})} is zero.\\

\item Clearly, the lower left
entry coming from {\bf line (\ref{equality4})} is $o(\phi)$.\\

\end{itemize}

Summarizing, from formulae (\ref{equality0}-\ref{equality4}) we obtain:
\begin{equation}\label{eqstep5}
\left(\dfrac{\lb^1_{21}}{\lb_{22}}(f(s,\phi)) -
\dfrac{\lb^1_{21}}{\lb_{22}}(s,\phi)\right) = \Delta(s) \phi + o(\phi).
\end{equation}

\vspace{20 pt}

\noindent \underline{{\sc Step 6}}.
Let $s=s^*$ and $\phi=1/n$ for a large enough $n\in \Z_+$. Denote
$[x]$ the integer part of $x$, {\it i.e.} $[x]=x-\{x\}$, and
$\pi_s(s,\phi)=s,\ \pi_\phi(s,\phi)=\phi$ the natural projections. Then
there exist constants $C,D>0$ independent of $n$ such that
\be \label{apriori-bound}
\pi_s f^{k}(s^*,\phi)-s^*<\tau, \quad
\frac{1}{Dn}<\pi_\phi f^{k}(s^*,\phi)<\frac{ D}{ n}
\quad \text{ for any }k\in [0,[\tau n/C]].
\ee
To see validity of these estimates switch to Lazutkin
coordinates, where we have (\ref{lazutkin-billiard-map}).
If $y=O(1/n)$, then
$$
f_L(x,y)=(x+y+O(1/n^3),y+O(1/n^4)).
$$
If we take $O(n)$ iterates, then $y$ (resp. $x$) component
can't change by more that $O(1/n^3)$ (resp. $O(1/n)$).
Then we can take the inverse map $L_f^{-1}$.

Combining all the estimates (\ref{equality0}-\ref{equality5})
we have that
\begin{equation}
\dfrac{\lb_{21}^1}{\lb_{22}}(f^{k+1}(s^*,\phi))=
\dfrac{\lb_{21}^1}{\lb_{22}}(f^k(s^*,\phi))
+ (\pi_\phi f^k(s^*,\phi))\,\Delta(\pi_s\ f^k(s^*,\phi))
+o(\phi).
\end{equation}

Thus, for each $k\in [0,[\tau n/C]]$ we have
\begin{eqnarray}
&& \dfrac{\lb_{21}^1}{\lb_{22}}(f^{[\tau n/C]}(s^*,\phi)) \ - \ \dfrac{\lb_{21}^1}{\lb_{22}}(s^*,\phi)\ = \nonumber\\
&&\qquad  =\ \sum_{k=0}^{[\tau n/C]-1}
\left((\pi_\phi f^k(s^*,\phi))\Delta(\pi_s\ f^k(s^*,\phi))
+o(\phi)\right).
\end{eqnarray}
Observe that $\lb_{11}(s)$ and $\lb_{22}(s)$ are strictly positive and
bounded from above by some $\gm>0$.
Thus using the above estimates and (\ref{sigmatau}), we conclude that the right-hand side of the above equality is bounded from below by
\[
\left[\frac{\tau n}{C}\right]
\ \frac{\phi}{D}\ \dfrac{\sigma}{\gm+2\eps}\ge \dfrac{\tau \sigma}{4CD \gm}>0.
\]
This implies that
\[
\dfrac{\lb_{21}^1}{\lb_{22}}
(f^{[\tau n]}(s^*,\phi)) - \dfrac{\lb_{21}^1}{\lb_{22}}((s^*,\phi))>\dfrac{\tau \sigma}{4DC \gm}>0.
\]
All the constants involved in this ratio are uniformly bounded.
This contradicts $\Lb_1\to 0$ as $\phi\to 0$.\\

\vspace{20 pt}

\noindent \underline{{\sc Step 7}}.
It follows from the above discussion that  $\Delta(s)\equiv 0$. Let us show that this implies that
$\Om_f$ and $\Om_g$ are similar. In fact

\begin{equation}\label{eqstep6}
\Delta(s)\equiv 0 \qquad \Longleftrightarrow \qquad \frac{d}{ds}\Big( \dfrac{1}{\rho_f^2(s)} \Big) \ = \
\frac{d}{ds}\Big( \dfrac{1}{\rho_g^2(\hs(s))}  \Big) \qquad \mbox{for all} \ s.
\end{equation}

Therefore, integrating we obtain that there exists $\a\in \R$ such that:
$$
 \rho_g^{-2}(\hs(s)) - \rho_f^{-2}(s) = \a,
$$
where $\a$ is a non-negative constant (otherwise just exchange the role of $f$ and $g$).

The above equality implies
$$
 \rho_g^{-1}(\hs(s)) = \sqrt{\a + \rho_f^{-2}(s)}.
$$

Integrating we obtain (up to rescaling one of the domain, we can assume that $C=1$):

\begin{eqnarray*}
2\pi &=& \int_0^{l_g} \rho_g (\hs) d\hs = \\
&=& \int_0^{l_f} \rho_g (\hs(s)) \frac{d\hs}{ds} ds =\\
&=& \int_0^{l_f} \rho_g (\hs(s)) \lb^{-2} ds =\\
&=& \int_0^{l_f} \rho_g^{5/3} (\hs(s)) \rho_f^{-2/3}(s) ds =\\
&=& \int_0^{l_f} \rho_f^{-2/3}(s) \left(\a + \rho_f^{-2}(s) \right)^{-5/6} ds =\\
&=& \int_0^{l_f} \rho_f^{-2/3}(s)  \rho_f^{5/3}(s) \left(1+ \a \rho_f^{2}(s) \right)^{-5/6} ds =\\
&=& \int_0^{l_f} \rho_f (s) \left(1+ \a \rho_f^{2}(s) \right)^{-5/6} ds =\\
&=& \int_0^{l_f} \rho_f (s) ds + \int_0^{l_f} \rho_f (s) \Big[\left(1+ \a \rho_f^{2}(s) \right)^{-5/6} -1\Big] ds =\\
&=& 2\pi + \int_0^{l_f} \rho_f (s)\Big[\left(1+ \a \rho_f^{2}(s) \right)^{-5/6} -1\Big] ds\,.
\end{eqnarray*}

Therefore
$$
\int_0^{l_f} \rho_f (s)\Big[\left(1+ \a \rho_f^{2}(s) \right)^{-5/6} -1\Big] ds = 0.
$$
Observe that
$
\Big[\left(1+ \a \rho_f^{2}(s) \right)^{-5/6} -1\Big] ds \leq 0
$
and
$\rho_f (s)>0$.
Hence it follows that  $\left(1+ \a \rho_f^{2}(s) \right)^{-5/6} \equiv 1$ and consequently $\a=0$.\\
\end{proof}

Now we prove that without the smoothness assumption, Question \ref{q2} would have a negative answer.

\begin{proposition}\label{C0conjugacy}
{\rm (i)} Let $\Om_f=\mathbb S^1$ be the circle  and $\Om_g$
be an ellipse of non-zero eccentricity. Then near the boundary
their billiard maps are $C^0$-conjugate.

{\rm (ii) }Similarly, if $\Om_f$ and $\Om_g$
are  ellipses of different eccentricities, then
their billiard maps are $C^0$-conjugate near the boundary.

{\rm (iii) } Let $\Om_f$ and $\Om_g$ be domains whose billiard maps
are integrable near the boundary. Then near the boundary
their billiard maps are $C^0$-conjugate.
\end{proposition}

\begin{proof} We only  prove the first  claim. Claim (ii) follows from (i), and claim (iii) can be proven similarly, proving that an integrable billiard is $C^0$ conjugate to a circular one, close to the boundary.

Let us prove (i). The billiard map has an analytic first integral,
i.e. an analytic function $F(s,\phi)$ such that
$F(s,\phi)=F(f(s,\phi))$ for all $(s,\phi)\in \A$.
This first integral divides the annulus $\A$ into
two regions separated by two separatrices. We are
interested only in the outer region outside of
separatrices. The outer region is foliated by analytic
curves, which correspond to elliptic caustics.
%The inner region is foliated by analytic curves
%which correspond to hyperbolic caustics.
Topologically the annulus has the form of
the phase space of the pendulum (see {\it e.g.}  Siburg \cite{Siburg}).

On each leave of this foliation orbits has the same
rotation number. This implies that conjugacy $h$ sends
a leave on $\Om_f$ with rotation number $\om$ to
a leave on $\Om_g$ of the same rotation number.
It turns out that one each leave there is a natural
parametrization so that billiard dynamics is
a rigid rotation by $\om$.

%If one maps the symmetry of one billiard to
%the symmetry axis of the on other, then the aforementioned
%parametrization defines the conjugacy on the whole
%annulus $\A$. Now we introduce the first integral
%and paramatrization each leave.

In the case of an ellipse there are explicit formulas.
Let the boundary curve be an ellipse, and parametrize it by
\[
\th \mapsto (x,y), \qquad \text{ where } x=h \cosh \mu_0 \cos \th,
\quad y=h \sinh \mu_0 \sin \th
\]
for $0\le \th <2\pi$. Let $s=s(\th)$ be arc-length parametrization
of the boundary with $s(0)=0$. Since it is smooth and non-degenerate
$\th=\th(s)$ is also well-defined. Then the billiard map has a first
integral
\[
I(\th,\phi)=\cosh^2 \mu_0 \cos^2 \phi + \cos^2 \th \sin^2 \phi.
\]
Namely, for any $(s,\phi)\in \A$ we have $f(s,\phi)=(s',\phi')$
and $I(\th(s),\phi)=I(\th(s'),\phi')$.

%A,
Parametrization on each caustic uses elliptic integrals and is fairly
cumbersome. Since we do not use it, we refer to
Tabanov \cite{Taba}.\\
%To determine parametrization define the special function
%\[
%F(z,k)=\int_0^z\dfrac{dt}{\sqrt{1-k^2\sin^2 t}}\]
%and introduce
%\[ \text{am}\,(F(z,k),k)=z\qquad \text{sn}\,(F(z,k),k)=\sin z
%\]
%which is the corresponding inverse function and its sine, respectively.

\end{proof}

%%%%%%%%%%%%%%%%%%%%%%%%%%%%%%%%%%%%%%%%%%%%

\appendix

%%%%%%%%%%%%%%%%%%%%
\section{Extension problem and Lazutkin invariant KAM curves}\label{extension}

In this section we want to state an {\it Extension problem} and discuss its relation to both Birkhoff conjecture and Guillemin-Melrose's one.

\subsection{Cantor set conjugacy}

Let us start by reproducing a version of a famous result of Lazutkin \cite{Lazutkin}.
It says that for any sufficiently smooth strictly convex domain $\Om$ the billiard
map has a Cantor set of convex invariant curves (caustics) that accumulate to
the boundary $\partial \Om$.

\bthm \label{Kovachev-Popov}
{\rm (Kovachev-Popov \cite{KP} using P\"oschel \cite{Poschel})}
Let $\Om$ be a smooth strictly convex
domain of unitary boundary length. Then there is a symplectic change
of coordinates $\Phi_f:(s,\phi)\to (\th, I)$ preserving the boundary
$\{\phi=0\}  \longleftrightarrow \{I =0\}$ near the boundary
such that the billiard map
$F=\Phi_f\circ f \circ \Phi_f^{-1}: (\th, I) \to (\th',I')$
is generated by
\[
S(\th,I')=\th I'+ K(I')^{3/2}+R(\th,I'),
\]
\be
\left\{
\beal
I    & = & \partial_\th S & = & I'+\partial_{\th} R(\th,I')\qquad \qquad \qquad  \\
\th' & = & \partial_{I'} S   & = & \th +
\frac 32 K(I')^{1/2}\, K'(I')+ \partial_{I'} R,
\enal \right.
\ee
where $K\in C^\infty(\R,\R)$ with $K(0)=0,\ K'(0)>0,$ and
$R\in C^\infty(\R^2,\R)$ is $1$-periodic in the first variable.
Moreover, there exists a Cantor set $\mathcal C^{\eps_*}\subset [0,\eps_*]$
with $0\in \mathcal C^{\eps_*}$, where $\eps_*$ is some small number
such that $R\equiv 0$ on $\R \times \mathcal C^{\eps_*}$.

Moreover, for any $N\in \Z_+$ there is $C=C_N$ with the property
the set $\mathcal C^{\eps_*}$ can be chosen so
that for any $0<\eps<\eps_*$ we have
\[
\eps-\text{Lebesgue measure }(\mathcal C^{\eps_*})<C \eps^N.
\]
\ethm

\vspace{10 pt}

\begin{remark}
Since the perturbation term $R$ vanishes on $R \times \mathcal C^{\eps_*}$
with all its derivatives, then each curve $\R \times \{I\},\ I \in \mathcal C^{\eps_*}$
gives rise to an invariant curve for the billiard map.\\
\end{remark}

Denote now $\mathcal C^{\eps_*}_f=\Phi_f^{-1} \mathcal C^{\eps_*}$
(resp. $\mathcal C^{\eps_*}_g=\Phi_g^{-1} \mathcal C^{\eps_*}$)
if the corresponding billiard map is $f$ (resp. $g$).

\bcor \label{main-corollary}
Let $f$ and $g$ be $C^r$ smooth billiard maps corresponding,
respectively, to strictly convex planar domains  $\Om_f$ and $\Om_g$,
with $r\ge 2$. Suppose that $h$ is a $C^{1,\a}$ conjugacy
for the restriction $f|_{\mathcal C_f}$ and $g|_{\mathcal C_g}$
with $\a>1/2$,  {\it i.e.} there exists $\eps_*>0$ such that
$f=h^{-1} {g} h$ for any $(s,\phi)\in \mathcal C_f^{\eps_*}$.
Then the two domains are similar, that is they are the same up to
a rescaling and an isometry.
\ecor

As observed in Remark \ref{remmaithm}, the proof of Theorem \ref{maintheorem} works also in the setting of Corollary \ref{main-corollary}. It  suffices to be able
to approach any point on the boundary .\\

\subsection{Mather's $\beta$ and $\a$-functions}\label{beta}
Let $0<p<q,\ p,q\in\Z_+$ be a pair of integers. Consider periodic orbits
of period $q$ of the billiard map having $p$ turns about the ellipse.
See Figure \ref{circle-billiard}: the dard red has $q=3,\ p=1$ (or $4$
depending on the orientation), the violet has $q=5,\ p=1$ (or $4$),
the green has $q=6,$\ $p=1$ (or $5$). By Birkhoff's
theorem \cite{Birkhoff} (see also \cite{Siburg}) there are always at least two. Pick one among
all of them the one of maximal perimeter. Denote it by $\L(p,q)$.
Define $\beta(p/q)=-\L(p,q)/q$. This function is well-defined for
all rational numbers in $[0,1]$. Moreover, it is strictly convex
as it was shown by Mather. Thus, it can be extended to all of $[0,1]$.

For twist maps this function was introduced by Mather and it is usually called {\it $\beta$-function}. Let us consider its convex conjugate (known as {\it Mather's $\a$-function}):

\[
\a(c)= - \max_\om (c\cdot \om - \beta(\om)).
\]

Since $\beta$ is convex, this is well-defined.

It turns out that
\[
\a(1+I)=K(I)^{3/2}.
\]
Siburg \cite{Siburg} shows that for a strictly convex $\Om$
with Lazutkin perimeter $C_\Om$ we have
\[\beta(\omega)= -\omega +\frac{C_\Om^3}{24} \om^3 + O(\om^5)
\]
Then:
$$
I(\omega) =  \partial \beta(\omega) =  \frac{C_\Om^3}{8} \om^2 + O(\om^4).
$$
Therefore
$$
\omega(I) =  2\sqrt{2} C_\Om^{-3/2}I^{1/2} + O(I).
$$
This implies the following expression for $\alpha$:
\begin{eqnarray*}
\alpha(1+I) &=& \omega(I)\cdot (1+I) - \beta(\omega(I)) = \\
&=& 2\sqrt{2} C_\Om^{-3/2}I^{3/2} +  O(I^2) - \frac{C_\Om^3}{24} \Big(2\sqrt{2} C_\Om^{-3/2}I^{1/2} + O(I)\Big)^3 + O(I^2) =\\
&=& 2\sqrt{2} C_\Om^{-3/2}I^{3/2}    - \frac{C_\Om^3}{24} \Big(8\sqrt{8} C_\Om^{-9/2}I^{3/2}\Big) + O(I^{5/2}) =\\
&=& \frac{4\sqrt{2}}{3} C_\Om^{-3/2}I^{3/2}   + O(I^{5/2}).\\
\end{eqnarray*}
Therefore:
\begin{eqnarray*}
K(I) &=& \big( \alpha(1+I)\big)^{2/3} = \Big(\frac{4\sqrt{2}}{3}\Big)^{2/3} C_\Om^{-1}I   + O(I^{2}).
\end{eqnarray*}

This implies that locally the corresponding twist map from
Theorem \ref{Kovachev-Popov} without the remainder term $R$
has the form
\be
\left\{
\beal
I'   & = &  I+O(I^{5/2})\qquad \qquad \\
\th' & = & \th + c_\Om I^{1/2} + O(I^{3/2}),
\enal \right.
\ee
where $c_\Om=2\sqrt{2}C_\Om^{-3/2}$. Here is the derivation of the second line:
%A, what do you mean by power smoething
%isn't, it 3/2
\begin{eqnarray*}
\theta' &=& \theta + \frac{3}{2}K(I')^{1/2} K'(I')
% + O(I^{\rm something}) = \\
 + O(I^{3/2}) = \\
&=& \theta +   \frac{3}{2} \Big[ \Big(\frac{4\sqrt{2}}{3}\Big)^{2/3} C_\Om^{-1}I   + O(I^{2})\Big]^{1/2} \Big[
\Big(\frac{4\sqrt{2}}{3}\Big)^{2/3} C_\Om^{-1}   + O(I) \Big] = \\
&=& \theta + 2\sqrt{2}C_\Om^{-3/2} I^{1/2} + O(I^{3/2}).
\end{eqnarray*}

Another way to see this formula is using Lazutkin coordinates
(\ref{Lazutkincoord}) and the form of the billiard map
in these coordinates (\ref{lazutkin-billiard-map}). Then notice
that this map preserves a {\it non--standard area form}:
$\dfrac{C_\Om^3}{8}dx\,dy^2$ (see for instance  \cite[Theorem 3.2.5]{Siburg}).
Letting $I=\dfrac{C_\Om^3}{8} y^2,\ x=\th$
we get the above expression.

\vskip 0.2in

\subsection{Extension Problem}
We can now state our {\it Extension Problem}.\\

\noindent {\bf Extension problem.}\ {\it Suppose we have two twist maps of
the form whose $\a$-functions coincide. Then corresponding
twist maps restricted to the Cantor set $\mathcal C$ are
$C^{1,\a}$-conjugate near the boundary for some $a>1/2$.}\\

\vskip 0.1in

Notice that the Cantor set $\mathcal C$ is the same for both
maps, because they have the same function $K$.\\

As remarked in the Introduction, a positive answer to this problem would allow to give a positive answer to Question \ref{q3} (Guillemin-Melrose' conjecture) as well as to Birkhoff's conjecture (see Conjecture \ref{conj}).

%%%%%%%%%%%%%%%%%%%%%%%%%%%%%%%%%%%%%%%%%%
%%%%%%%%%%%%%%%%%%%%%%%%%%%%%%%%%%%%%%%%

\section{Proof of Proposition \ref{propformbillmap}} \label{appendix}

\begin{proof}
If the boundary of the billiard is $C^k$, then the billiard map is at
least $C^{k-1}$ (see for instance \cite[Theorem 4.2 in Part V]{KatokStrelcyn}).
Moreover  it is possible to show (see \cite[Theorem 4.2 in Part V]{KatokStrelcyn}) that:

\begin{equation}\label{formulae}
\left\{\begin{array}{ll}
\partial_ss' = \dfrac{\rho(s)\ell(s,s') - \sin \phi}{\sin \phi'}
&
\partial_\phi s' = \dfrac{\ell(s,s')}{\sin \phi'}\\
&\\
\partial_s\phi' = \dfrac{\rho(s)\rho(s')\ell(s,s') -  \rho(s)\sin\phi' - \rho(s') \sin \phi} {\sin \phi'}
&
\partial_\phi\phi' = \dfrac{\rho(s')\ell(s,s') - \sin \phi'}{\sin \phi'}
\end{array}\right.
\end{equation}
where $\ell(s,s')$ denotes the distance in
the plane between the points on the boundary of
the billiards corresponding to $s$ and $s'$.\\

In particular, it is easy to check that  $\partial_ss'(s,0)=1$. In fact the billiard map
coincides with the identity on the boundary $\{\f=0\}$, {\it i.e.} $s'(s,0)=s$ for all $s$.
Similarly, $\partial_s \f'(s,0) \equiv 0$.\\

Let us prove now that  $\partial_\f s'(s,0)=\dfrac{2}{\rho(s)}$. In fact, applying de
l'H\^opital formula, using formulae
(\ref{formulae}) and noticing that
$\partial_{s'}\ell(s,s')= \cos \phi'$, we obtain:
\begin{eqnarray}\label{limitL}
L &:=& \lim_{\phi\rightarrow 0^+} \dfrac{\ell(s,s')}{\sin \phi'} \;=\;
\lim_{\phi\rightarrow 0^+} \dfrac{\partial_{s'}\ell(s,s') \partial_{\phi}s'}{\cos \phi' \partial_\phi \phi'}  =\nonumber\\
&=& \lim_{\phi\rightarrow 0^+} \dfrac{\partial_{\phi}s'}{\partial_\phi \phi'}  \;=\;
\lim_{\phi\rightarrow 0^+} \dfrac{ \dfrac{\ell(s,s')}{\sin \phi'}}
{\dfrac{\rho(s')\ell(s,s') - \sin \phi'}{\sin \phi'}}  =\nonumber\\
&=& \dfrac{L}{\rho(s) L - 1}\,.
\end{eqnarray}

It follows from the fact the curve is strictly convex and \cite[Theorem 4.3 in Part V]{KatokStrelcyn} that $L<+\infty$. We want to show that $L>0$. In fact, let us consider the osculating circle at point $\g(s') \in \partial \Omega$ with radius $\frac{1}{\rho(s')}$. Elementary geometry shows that
$\ell(s,s') \geq  \dfrac{2}{\rho(s')}\sin \phi'$. Therefore:
$$
L \geq \dfrac{2}{\max_{s\in[0,l]} \rho(s)} >0.
$$

This fact and (\ref{limitL}) allow us to conclude that
$$
L = \dfrac{L}{\rho(s) L - 1} \quad \Longleftrightarrow \quad \rho(s) L - 1 = 1
\quad \Longleftrightarrow \quad L = \dfrac{2}{\rho(s)}\,.
$$

Furthermore, $\partial_\f \f'(s,0) = 1$. In fact:
	\begin{eqnarray*}
	\partial_\f \f' (s,\f) &=& \dfrac{\rho(s')\ell(s,s') - \sin \phi'}{\sin \phi'} =\\
	&=&  \big(\rho(s) + O(\phi)\big) \dfrac{\ell(s,s')}{\sin \f'} - 1=\\
	&=& \big(\rho(s) + O(\phi)\big) \Big(\frac{2}{\rho(s)} + O(\f)\Big) - 1 =\\
	&=& 1 + O(\f) \stackrel{\f \rightarrow 0^+}{\longrightarrow} 1.
	\end{eqnarray*}

To complete the proof, we need to compute explicitly the higher-order terms.

\begin{itemize}
\item Observe that
\begin{eqnarray*}
&& \partial^2_{\f \f}\f'(s,\f) \ = \ \frac{\partial}{\partial\f} \left( \frac{ \rho(s')\ell(s,s') }{ \sin \f'} - 1 \right) =\\
&& \qquad = \ \frac{\big[\rho'(s')\sf \ell(s,s') + \rho(s')\partial_{s'}\ell(s,s')\sf \big] \sin \f' - \rho(s')\ell(s,s') \cos \f' \ff }{\sin^2 \f'} =\\
&& \qquad = \ \rho'(s')\sf \frac{\ell(s,s')}{\sin \f'} +
\rho(s') \sf \frac{\partial_{s'}\ell(s,s')}{\sin \f'}
 - \rho(s') \ff  \frac{\ell(s,s')}{\sin \f'}  \frac{\cos \f'}{\sin \f'} =\\
&& \qquad = \ 4  \frac{\rho'(s)}{\rho^2(s)} +
O(\f)-
\rho(s') \frac{\ell(s,s')}{\sin \f'}  \frac{\cos \f'}{\sin \f'} \Big( \partial^2_{\f\f}\f'(s,0)\f + O(\f^2) \Big) =\\
&& \qquad = \ 4  \frac{\rho'(s)}{\rho^2(s)} -
\Big(\rho(s) + O(\f)\Big) \Big(\frac{2}{\rho(s)}+O(\f)\Big)  \dfrac{1 + O(\f^2)}{\f + O(\f^3)} \Big( \partial^2_{\f\f}\f'(s,0)\f + O(\f^2) \Big) = \\
&& \qquad = \ 4  \frac{\rho'(s)}{\rho^2(s)} - 2 \partial^2_{\f\f}\f'(s,0) + O(\f)\,.
 \end{eqnarray*}
Therefore,
$$
\partial^2_{\f\f}\f'(s,0) = \frac{4}{3} \frac{\rho'(s)}{\rho^2(s)}.
$$

\item
Similarly:
\begin{eqnarray*}
\partial^2_{\f \f}s'(s,\f) &=& \frac{\partial}{\partial \f} \left( \frac{ \ell(s,s')}{ \sin \f'}\right) =\\
&=&
\frac{\partial_{s'}\ell(s,s')\sf \sin \f' - \ell(s,s') \cos \f' \ff }{\sin^2 \f'} =\\
&=&
\frac{\cos \f' }{\sin \f'} \sf   -
\frac{\ell(s,s')}{\sin \f'}  \frac{\cos \f'}{\sin \f'}  \ff=\\
&=&  - \frac{\ell(s,s')}{\sin \f'}  \frac{\cos \f'}{\sin \f'}  \Big[ \frac{4}{3} \frac{\rho'(s)}{\rho^2(s)} \f + O(\f^2)  \Big] =\\
&=& - \frac{8}{3} \frac{\rho'(s)}{\rho^3(s)} + O(\f).
\end{eqnarray*}

\item
\begin{eqnarray*}
&& \partial^2_{s\f} s'(s,\f) \ = \  \frac{\partial }{\partial\f} \left(  \dfrac{\rho(s)\ell(s,s') - \sin \phi}{\sin \phi'}        \right) =\\
&& \qquad = \ \dfrac{\Big(\rho(s) \partial_{s'}\ell(s,s') \partial_\f s'    - \cos \phi\Big) \sin \f' - \Big(\rho(s)\ell(s,s') - \sin \phi\Big) \cos \f' \partial_\f \f'}
{\sin^2 \phi'} =\\
&& \qquad = \
\dfrac{\rho(s) \cos \f' \ell(s,s')     - \cos \phi \sin \f' - \cos \f' \Big(\rho(s)\ell(s,s') - \sin \phi\Big)  \Big( 1 + \frac{4}{3} \frac{\rho'(s)}{\rho^2(s)}\f + O(\phi^2) \Big)}
{\sin^2 \phi'} =\\
&& \qquad = \
\dfrac{ - \cos \phi \sin \f' +  \cos \phi' \sin \f  - \cos \phi'  \frac{4}{3} \frac{\rho'(s)}{\rho(s)}\f  \Big(\rho(s)\ell(s,s') - \sin \f   \Big) + O(\f^3)}
{\sin^2 \phi'} =\\
&& \qquad = \ \dfrac{ - \cos \phi  +  \cos \phi' \frac{\sin \f}{\sin \f'}  - \cos \phi'  \frac{4}{3} \frac{\rho'(s)}{\rho(s)}\f  \Big(\rho(s)\frac{\ell(s,s')}{\sin \f'} - \frac{\sin \f}{\sin \f'} \Big)+ O(\f^2)}
{\sin \phi'} =\\
&& \qquad = \ \dfrac{  O(\f^2)    - \big(1 + O(\f^2)\big) \frac{4}{3} \frac{\rho'(s)}{\rho(s)}\f  \Big[\rho(s) \big(\frac{2}{\rho(s)}+ O(\f) \big) - 1 + O(\f^2) \Big]}
{\f+O(\f^3)} =\\
&& \qquad = \ -\frac{4}{3} \frac{\rho'(s)}{\rho(s)} + O(\f),
\end{eqnarray*}

where we used that
$$
\frac{\sin \f}{\sin\f'} = \frac{\f + O(\f^3)}{\f' + O({\f'}^3)} = \frac{\f + O(\f^3)}{\f + O({\f}^3)} = 1 + O(\f^2).
$$

Therefore,
$$
\partial^2_{s\f} s'(s,0)  = -\frac{4}{3} \frac{\rho'(s)}{\rho(s)}.
$$

\item Finally, using the above expressions we obtain:
\begin{eqnarray*}
&& \partial_s\phi' (s,\f) \ = \ \dfrac{\rho(s)\rho(s')\ell(s,s') -  \rho(s)\sin\phi' - \rho(s') \sin \phi} {\sin \phi'} =\\
&& \quad =\ \rho(s)\rho\big(s + \frac{2}{\rho(s)}\f + O(\f^2)\big) \frac{\ell(s,s')}{\sin \f'} -  \rho(s) - \rho\big(s + \frac{2}{\rho(s)}\f + O(\f^2)\big) \frac{\sin \phi}{\sin \phi'} =\\
&& \quad =\ \rho(s) \Big[\rho(s) + \rho'(s) \frac{2}{\rho(s)}\f + O(\f^2)\Big] \partial_\f s'(s,\f) -  \rho(s) -  \\&& \qquad \ \  - \Big[\rho(s) + \rho'(s) \frac{2}{\rho(s)}\f + O(\f^2)\Big]  \Big(1+O(\f^2)\Big) =\\
&& \quad =\ \rho(s) \Big[\rho(s) + \rho'(s) \frac{2}{\rho(s)}\f + O(\f^2)\Big] \Big( \frac{2}{\rho(s)}  -   \frac{8}{3} \frac{\rho'(s)}{\rho^3(s)} \phi +
O(\phi^2) \Big) -  \rho(s) - \\
&& \qquad \ \  - \Big[\rho(s) + \rho'(s) \frac{2}{\rho(s)}\f + O(\f^2)\Big]  \Big(1+O(\f^2)\Big) =\\
&& \quad =\  \frac{2 \rho'(s)}{\rho(s)}\f - \frac{8}{3} \frac{\rho'(s)}{\rho (s)}\f + O(\f^2) =\\
&& \quad =\ - \frac{2}{3} \frac{\rho'(s)}{\rho (s)} \f + O(\f^2).
 \end{eqnarray*}
\end{itemize}

\end{proof}

%%%%%%%%%%%%%%%%%%%%%%%%%%%%%%%%%%%%

\section{On Lazutkin coordinates}\label{appLaz}

As we have recalled in Section \ref{sec2}, in \cite{Lazutkin}  Lazutkin introduced a very special
change of coordinates that reduces the billiard map $f$ to a very simple form.

Let $L_f:  {[0,l]\times [0,\pi] \to  \T \times [0,\dt]}$  with small $\dt>0$ be given by
\begin{equation}\label{Lazutkincoord}
L_f(s,\phi)=\left(x=C^{-1}_f \int_0^s \rho^{-2/3}(s)ds,\qquad
y=4C_f^{-1}\rho^{1/3}(s)\ \sin \phi/2 \right),
\end{equation}
where $C_f := \int_0^l \rho^{-2/3} (s)ds$ is the so-called {\it Lazutkin perimeter}.

In these new coordinates the billiard map becomes very simple (see \cite{Lazutkin}):

\be \label{lazutkin-billiard-map}
f_L(x,y) = \Big( x+y +O(y^3),y + O(y^4) \Big)
\ee

In particular, near the boundary $\{\phi=0\} = \{y=0\}$, the billiard map $f_L$ reduces to
a small perturbation of the integrable map $(x,y)\longmapsto (x+y,y)$.\\

Let us now consider two strictly convex $C^4$ domains $\Om_f$ and $\Om_g$
and, as in the hypothesis of our Theorem \ref{maintheorem}, suppose that there exists a $C^1$ conjugacy $h$ in a neighborhood of the boundary, {\it i.e.} for some $\dt>0$
we have $f=h^{-1} {g} h$ for any $s\in \partial \Om_f$ and $0\le \phi\le \dt$.\\

Let $L_f$ and $L_g$ denote the Lazutkin change of coordinates for the billiard maps $f$ and $g$, $f_L$ and $g_L$ the corresponding maps in Lazutkin coordinates and  by $C_f$ and $C_g$ the respective Lazutkin perimeters.  We can summarise everything in the following commutative diagram:

\begin{equation}\label{diagram}
\xymatrix{
\T\times [0,\delta_f]   \ar@{->}[ddd]_{h_L}
\ar@{->}[rrr]^{f_L}    &&&
\T\times [0,\delta_f] \ar@{->}[ddd]^{h_L} \\
& \bar{\A}_f \ar@{->}[ul]^{L_f}  \ar@{->}[d]_{h}
\ar@{->}[r]^{f}    & \bar{\A}_f  \ar@{->}[ur]_{L_f}  \ar@ {->}[d]^{h} &\\
& \bar{\A}_g   \ar@{->}[dl]_{L_g}   \ar@{->}[r]_{g}
& \bar{\A}_g  \ar@{->}[dr]^{L_g} &\\
\T\times [0,\delta_g]  \ar@{->}[rrr]_{g_L}
&&& \T\times [0,\delta_g]\\
}
 \end{equation}

Let $h_L$ be the corresponding conjugacy between the billiard maps $f_L$ and $g_L$ in Lazutkin coordinates, {\it i.e.}
$h _L^{-1} \circ g_L \circ h_L = f_L$. \\

Morever, let $C:= C_fC_g^{-1}$ and $\lambda(s):= \dfrac{\rho_f^{1/3}(s)}{\rho_g^{1/3}(s)}$. We can prove the following lemmata, that have been used in the proof of our main results.\\

\begin{lemma}
If $\hat{s}(s)$ is defined by the relation $h(s,0)=(\hat{s},0)$, then:
$$\dfrac{d\hat{s}(s)}{ds} = C^{-1} \lb^{-2}.$$
\end{lemma}

\begin{proof}
Using (\ref{Lazutkincoord}) and the fact that $h_L(x,0)=(x,0)$
(zero points on both domains are marked so that $h(0,0)=(0,0)$), we obtain:

\begin{eqnarray}\label{hats}
\big(\hs(s+\dt),\ 0\big) &=& L_g^{-1}\Big( h_L \big(L_f(s+\dt,0)\big)\Big) \ = \nonumber\\
&=& L_g^{-1}\big(L_f(s+\dt,0)\big) \ = \nonumber\\
&=& L_g^{-1}\big(L_f(s,0)\big)  + DL_g^{-1}\big(L_f(s,0)\big) DL_f(s,0) \ \left(\begin{array}{l} \dt \\0\end{array}\right) \ = \nonumber\\
&=& \big(\hs,0\big)  + \big[DL_g(\hs, 0)\big]^{-1} DL_f(s,0) \ \left(\begin{array}{l} \dt \\0\end{array}\right) \ = \nonumber\\
&=& \big(\hs,0\big)  +
{\tiny \left( \begin{array}{ll}
C_g \rho_g^{2/3} & 0 \nonumber\\
0 &   \dfrac{1}{2}C_g\ \rho_g^{-1/3}
\end{array}
\right)
\left( \begin{array}{ll}
C_f^{-1} \rho_f^{-2/3} & 0 \\
0 &   2C_f^{-1}\ \rho_f^{1/3}
\end{array}
\right)
\left(\begin{array}{l} \dt \\0\end{array}\right)}  = \nonumber\\
&=& \big(\hs + C^{-1} \lb^{-2}\dt, \ 0). \nonumber
\end{eqnarray}
\end{proof}

\vspace{10 pt}

Let us now compute $Dh_L(x,0)$.\\

\begin{lemma}
$$
Dh_L(x,0) =
\left(\begin{array}{ll}
1 & * \\
0 & 1
\end{array} \right)
$$
\end{lemma}

\begin{proof}
It follows from diagram (\ref{diagram}) that $f_L = h_L^{-1} \circ g_L \circ h_L $ and therefore
$Df_L(x,0) = Dh_L(x,0)^{-1} \cdot Dg_L(x,0) \cdot Dh_L(x,0)$, where we used that $h_L(x,0)=(x,0)$.

Let us denote
$$Dh_L(x,0) = \left(\begin{array}{ll}
a & b \\
c & d
\end{array} \right)$$
and let $D:= ad-bc$ be its determinant. A-priori, all these entries might depend on $x$. Then:

\begin{eqnarray*}
\left(\begin{array}{ll}
1 & 1 \\
0 & 1
\end{array} \right) &=& \dfrac{1}{D}
\left(\begin{array}{ll}
d & -b \\
-c & a
\end{array} \right) \cdot
\left(\begin{array}{ll}
1 & 1 \\
0 & 1
\end{array} \right) \cdot
\left(\begin{array}{ll}
a & b \\
c & d
\end{array} \right)  = \\
&=& \dfrac{1}{D} \left(\begin{array}{cc}
ad + dc - bc & d^2 \\
-c^2 & -cb -cd + ad
\end{array} \right)\,.
\end{eqnarray*}

It follows that $c=0$ and $D=ad$. Therefore, the above equality becomes:
\begin{eqnarray*}
\left(\begin{array}{ll}
1 & 1 \\
0 & 1
\end{array} \right) &=&
\left(\begin{array}{ll}
1 & \frac{d}{a} \\
0 & 1
\end{array} \right)
\end{eqnarray*}
and consequently $a=d$ for all $x$. Observe that actually $a(x)\equiv 1$, since $h_L(x,0)=(x,0)$.
In conclusion,
$$
Dh_L(x,0) =
\left(\begin{array}{ll}
1 & * \\
0 & 1
\end{array} \right)
$$
\end{proof}

%%%%%%%%%%%%%%%%%%%%%%%%%%%%%%%%

\end{document}